\documentclass{amsart}
\usepackage{amsmath, amscd, amssymb, amsfonts, stmaryrd, vmargin, amsthm, exscale, relsize, mathrsfs, url, comment}

\theoremstyle{plain}        \newtheorem{thm}{Theorem}

\theoremstyle{plain}        \newtheorem{pro}[thm]{Proposition}

\theoremstyle{plain}        \newtheorem{lem}[thm]{Lemma}

\theoremstyle{plain}        \newtheorem{cor}[thm]{Corollary}

\theoremstyle{plain}        \newtheorem{rem}{Remark}

\theoremstyle{plain}        

\theoremstyle{plain}        \newtheorem{defn}{Definition}

\theoremstyle{plain}        \newtheorem{Problem}{Problem}

\theoremstyle{plain}        \newtheorem{eg}{Example}

\begin{document}
\title[Local holomorphic maps preserving $(p,p)$-forms]
{On local holomorphic maps between K\"ahler manifolds preserving $(p,p)$-forms}

\author{Shan Tai Chan}
\address{Institute of Mathematics\\
Academy of Mathematics and Systems Science\\
Chinese Academy of Sciences\\
Beijing 100190\\
China}
\email{mastchan@amss.ac.cn}

\author{Yuan Yuan}
\address{Department of Mathematics, Syracuse University, Syracuse, NY 13244, USA.}
\email{yyuan05@syr.edu}

\subjclass[2010]{32H02 (primary); 32Q15, 53C24 (secondary)}

\begin{abstract}
We study local holomorphic maps between K\"ahler manifolds preserving $(p,p)$-forms.
In this direction, we prove that any such local holomorphic map $F$ is a holomorphic isometry up to a scalar constant provided that $p$ is strictly less than the complex dimension of the domain of $F$.
We then study local holomorphic maps between finite dimensional complex space forms preserving invariant $(p,p)$-forms.
It was proved by Calabi that there does not exist a local holomorphic isometry between complex space forms $M$ and $N$ provided that $M$ and $N$ are of different types.
In this article, we generalize this result to local holomorphic maps between complex space forms $M$ and $N$ preserving invariant $(p,p)$-forms whenever $M$ and $N$ are of different types except for the case where the universal covers of $M, N$ are biholomorphic to $\mathbb{C}^m, \mathbb{P}^n$, respectively and $2\le p=m<n$. We also obtain some results in more general settings, including the study on indefinite K\"ahler manifolds and relatives for K\"ahler manifolds.
\end{abstract}

\maketitle

\section{Introduction}
The existence and characterization of holomorphic isometric embeddings are classical problems in complex and differential geometry.
In \cite{Ca53}, Calabi obtained the fundamental results of existence, uniqueness and global extension of a local holomorphic isometric embedding from a K\"ahler manifold into a complex space form. In particular, Calabi proved that there does not exist a local holomorphic isometric embedding between complex space forms of different types. Umehara \cite{Um88} later generalized Calabi's non-existence result for holomorphic isometries to indefinite complex space forms. Moreover, following Calabi's idea, Umehara \cite{Um87} proved that two complex space forms of different types cannot share a common K\"ahler submanifold. According to Di Scala-Loi \cite{DL10}, two K\"ahler manifolds are called relatives if they share a common K\"ahler submanifold. Furthermore, it was proved in \cite{DL10, HY15} that two Hermitian symmetric spaces of different types are not relatives. The relativity problem has recently been extensively studied by many authors (cf. \cite{LM22, ZJ22} and reference therein). The interested reader may refer to \cite{LZ18} for the various problems on holomorphic isometric immersions into complex space forms.

On the other hand, in order to study the commutants of Hecke correspondences, Clozel-Ullmo \cite{CU03} considered the germs of holomorphic maps arising from an algebraic correspondence between the quotient an irreducible bounded symmetric domain $\Omega$ by a torsion-free, discrete group of automorphisms. They are able to reduce the problem of the characterization of modular correspondences to the total geodesy of germs of holomorphic isometries or measure-preserving holomorphic maps from $\Omega$ to $\Omega \times \cdots \times \Omega$. The total geodesy of holomorphic isometries arisen from algebraic correspondences was proved by Mok \cite{Mo02, Mo12} (cf. \cite{Mo18, YZ12, CXY17, Yu19, LM21} and references therein for general problems of holomorphic isometries). For measure-preserving holomorphic maps, this  problem was solved by Mok-Ng \cite{MN12}. The reader may also refer to \cite{HY14, FHX20} for the problem of holomorphic isometries and measure-preserving maps between Hermitian symmetric spaces of compact type. Moreover, the problem of characterization of modular correspondences considered by Clozel-Ullmo can also be reduced to the total geodesy of germs of holomorphic maps from $\Omega$ to $\Omega \times \cdots \times \Omega$ preserving invariant $(p, p)$-forms and this general problem of total geodesy was raised by Mok \cite{Mo11} and solved partially by the second author \cite{Yu17}. 

Generally speaking, holomorphic maps preserving $(p, p)$-forms can be considered as natural generalizations of holomorphic isometries and holomorphic measure-preserving maps. Many fundamental problems can be raised, such as (non)-existence and rigidity. This paper is devoted to the study in this direction and the main result is the generalization of the non-existence of holomorphic isometries between complex space forms of different types by Calabi \cite{Ca53}.

We first study local holomorphic maps between K\"ahler manifolds preserving $(p,p)$-forms and observe the following metric rigidity result.

\begin{pro}\label{thm:pp_form_to_Holo_iso_0}
Let $(M,\omega_M)$ and $(N,\omega_N)$ be K\"ahler manifolds of finite complex dimensions $m$ and $n$ respectively, where $\omega_M$ and $\omega_N$ denote the corresponding K\"ahler forms.
Let $F:(M;x_0)\to (N;F(x_0))$ be a germ of holomorphic map such that $F^*\omega_N^p = \lambda \omega_M^p$ for some real constant $\lambda>0$ and some integer $p$, $1\le p\le \min\{m,n\}$. Then, $m\le n$.

If we assume in addition that $1\le p<m$, then $F^*\omega_N=\lambda^{1\over p} \omega_M$ so that $F:(M,\lambda^{1\over p} \omega_M;x_0)\to (N,\omega_N;F(x_0))$ is a local holomorphic isometry.
\end{pro}

For any bounded domain $U\Subset \mathbb C^m$ we write $\omega_U$ for the K\"ahler form of $(U,ds_U^2)$, where $ds_U^2$ denotes the Bergman metric on $U$. As a consequence of Proposition \ref{thm:pp_form_to_Holo_iso_0} and some basic observations, we have the following rigidity theorem for local holomorphic maps between bounded symmetric domains of higher rank preserving invariant $(p,p)$-forms, which solves Problem 5.3.1 in Mok \cite{Mo11} partially.

\begin{thm}\label{thm:local holo map_BSD_pp_form}
Let $D\Subset \mathbb C^n$ and $\Omega\Subset \mathbb C^N$ be finite dimensional bounded symmetric domains equipped with their Bergman metrics $ds_D^2$ and $ds_\Omega^2$ respectively.
Let $f:(D;x_0)\to (\Omega; f(x_0))$ be a germ of holomorphic map such that $f^*\omega_{\Omega}^p = \lambda \omega_D^p$ for some real constant $\lambda>0$ and some integer $p$, $1\le p\le \min\{n,N\}$.

If $1\le p < n $ or $p=n=N$, then $f$ extends to a proper holomorphic isometric embedding $F:(D,\lambda^{1\over p} ds_D^2) \to (\Omega,ds_\Omega^2)$.
Moreover, if $p=n=N$, then we have $\lambda=1$.

If we assume in addition that each irreducible factor of $D$ is of rank $\ge 2$, then $F$ is totally geodesic, and so is $f$.
In general, $f$ is totally geodesic provided that any holomorphic isometry from $(D,\lambda^{1\over p} ds_D^2)$ to $(\Omega,ds_\Omega^2)$ is totally geodesic.
\end{thm}
\begin{rem}
As an example, the rigidity theorem of Yuan-Zhang \cite{YZ12} yields the total geodesy of $F$ when $D=\mathbb B^n$, $n\ge 2$, and $\Omega$ is a product of complex unit balls.
\end{rem}

Let $(X, \omega_X)$ be a complex space form, i.e., a connected complete K\"ahler manifold of constant holomorphic sectional curvature.
We write $X=X(c)$ to indicate that $X$ is of constant holomorphic sectional curvature $c$.
Then, the universal covering space of $X$ is either a complex Euclidean space, a complex unit ball or a complex projective space.
Let $M$ and $N$ be  complex space forms.
We say that $M$ and $N$ are of the same type if and only if $M=M(c)$ and $N=N(c')$ such that either $c$ and $c'$ are of the same sign or $c=c'=0$.
We also say that $M$ and $N$ are of different types if and only if $M$ and $N$ are not of the same type. Motivated by the work of Calabi \cite{Ca53}, we study local holomorphic maps between complex space forms preserving invariant $(p,p)$-forms, and our main result is the following non-existence theorem.

\begin{thm}\label{thm:CSF diff type_non-exist_holo map_pp forms}
Let $(M,\omega_M)$ and $(N,\omega_N)$ be complex space forms of complex dimensions $m$ and $n$ respectively.
If $M$ and $N$ are of different types, then for any real constant $\lambda>0$ and any integer $p$, $1\le p\le \min\{m,n\}$, there does not exist a local holomorphic map $F:(M;x_0) \to (N;F(x_0))$ such that $F^*\omega_N^p = \lambda \omega_M^p$ holds except for the case where the universal cover of $M$ is biholomorphic to $\mathbb{C}^m$, $N= \mathbb{P}^n$ and $2\le p=m<n$.
\end{thm}

In the settings of Theorem \ref{thm:CSF diff type_non-exist_holo map_pp forms}, we write $\pi_M: \widetilde M\to M$ and $\pi_N:\widetilde N\to N$ for the universal covering maps of $M$ and $N$ respectively.
We have the K\"ahler metrics $\omega_{\widetilde M}=\pi_M^*\omega_M$ and $\omega_{\widetilde N}=\pi_N^* \omega_N$ on the universal covering spaces $\widetilde M$ and $\widetilde N$ of $M$ and $N$ respectively.
If there exists a local holomorphic map $F:(M;x_0) \to (N;F(x_0))$ such that $F^*\omega_N^p = \lambda \omega_M^p$ holds on the domain of $F$, then we would have a local holomorphic map $\widetilde F:(\widetilde M;\widetilde x_0) \to (\widetilde N;\widetilde F(\widetilde x_0))$ such that $\widetilde F^*\omega_{\widetilde N}^p = \lambda \omega_{\widetilde M}^p$ holds locally around $\widetilde x_0$ by a local lifting of $F$ to the universal covering spaces.
Hence, it suffices to prove Theorem \ref{thm:CSF diff type_non-exist_holo map_pp forms} in the case where $M$ and $N$ are simply connected.

\medskip
\noindent{\bf Acknowledgement}: Part of the work was done when the second author was visiting BICMR in Spring 2022 and he would like to thank the center for providing him the wonderful research environment.
We would also like to thank the referee for helpful comments.

\section{Preliminary}
Given a K\"ahler manifold $M$ of complex dimension $m$, we write $\omega_M$ for its K\"ahler form, and $g_M$ for the corresponding K\"ahler metric.
Fix $x_0\in M$ and let $w=(w_1,\ldots,w_m)$ be local holomorphic coordinates around $x_0$. For $1\le p\le m$, we write ${\bf A}=\{I=(i_1,\ldots,i_p): 1\le i_1<\cdots<i_p \le m\}$ and $\omega_M^p =  (\sqrt{-1})^p\sum_{I,J\in {\bf A}}(\omega_M^p)_{I\overline J} dw^I \wedge d\overline{w^J}$ in terms of the local holomorphic coordinates.
In what follows, for any $v\in \wedge^p T_{x_0}(M)$, we write $v=\sum_{I\in {\bf A}} v_{I} {\partial \over \partial w_{i_1}}\wedge\cdots \wedge {\partial \over \partial w_{i_p}}$, $I=(i_1,\ldots,i_p)$, and write
$\omega_M^p(x_0)(v,\overline v)
:= \sum_{I,J\in {\bf A}} (\omega_M^p)_{I\overline J}(x_0) v_{I} \overline{v_{J}}$.

For any complex manifold $X$, we write $T_X$ for the holomorphic tangent bundle of $X$, $\Theta_E$ for the curvature operator of the Chern connection on a Hermitian holomorphic vector bundle $(E,h)$ over $X$, and $\Lambda(X)$ for the Umehara algebra on $X$ as defined in Umehara \cite{Um88}.
More precisely, $\Lambda(X)$ is the associative algebra of real analytic functions on $X$ that consists of real linear combinations of functions $f\overline{g}+g\overline{f}$ for holomorphic functions $f$ and $g$ on $X$. 

\subsection{Complex space forms}\label{Sec:CSF}
Recall that complex space forms are defined as connected complete K\"ahler manifolds of constant holomorphic sectional curvatures, and there are three different types of simply connected complex space forms, namely, complex Euclidean spaces, complex unit balls, and complex projective spaces.

In the present article, we write $\mathbb C^n$ for the complex Euclidean space of complex dimension $n$ and $\mathbb P^n = (\mathbb C^{n+1}\smallsetminus\{{\bf 0}\})/\mathbb C^*$ for the complex projective space of complex dimension $n$, where $n\ge 1$ is an integer. Denote by
$\mathbb B^n:=\left\{z=(z_1,\ldots,z_n)\in \mathbb C^n: \lVert z\rVert^2:=\sum_{j=1}^n |z_j|^2<1\right\}$ the complex unit ball in $\mathbb C^n$. We also write $\Delta:=\{z\in \mathbb C: |z|^2<1\}$ for the open unit disk in the complex plane $\mathbb C$.

Let $\omega_{\mathbb C^n}:=\sqrt{-1}\partial\overline\partial \left(\sum_{j=1}^n |z_j|^2\right) = \sqrt{-1} \sum_{j=1}^n dz_j\wedge d\overline{z_j}$ be the standard K\"ahler form on the complex Euclidean space $\mathbb C^n$ such that $(\mathbb C^n,\omega_{\mathbb C^n})$ is of constant holomorphic sectional curvature $0$.
(Noting that the corresponding K\"ahler metric is the standard complex Euclidean metric on $\mathbb C^n$ up to a scalar constant.)
Let 
\[ \omega_{\mathbb B^n}:= -\sqrt{-1}\partial\overline\partial \log \left(1-\sum_{j=1}^n|z_j|^2\right)= {\sqrt{-1}\over (1-\lVert z \rVert^2)^2} \sum_{i,j=1}^n \left( (1-\lVert z \rVert^2)\delta_{ij}+\overline{z_i} z_j\right) dz_i\wedge d\overline{z_j}\]
be the standard K\"ahler form on $\mathbb B^n$ such that $(\mathbb B^n, \omega_{\mathbb B^n})$ is of constant holomorphic sectional curvature $-2$. 
We also let $\omega_{\mathbb P^n}$ be the standard K\"ahler form on $\mathbb P^n$ such that $(\mathbb P^n,\omega_{\mathbb P^n})$ is of constant holomorphic sectional curvature $2$. (Noting that the corresponding K\"ahler metric is the well-known Fubini-Study metric on $\mathbb P^n$ up to a scalar constant.)
More precisely, on the dense open subset $U_0=\{[\xi_0,\xi_1,\ldots,\xi_n]\in \mathbb P^n: \xi_0 \neq 0\}$ of $\mathbb P^n$, the standard K\"ahler form $\omega_{\mathbb P^n}$ on $\mathbb P^n$ can be written as
\[ \begin{split}
\omega_{\mathbb P^n}=& \sqrt{-1}\partial\overline\partial \log \left(1+\sum_{j=1}^n|w_j|^2\right)
= {\sqrt{-1}\over (1+\lVert w \rVert^2)^2} \sum_{i,j=1}^n \left( (1+\lVert w \rVert^2)\delta_{ij}-\overline{w_i} w_j\right) dw_i\wedge d\overline{w_j},
\end{split}\]
where $w_j:=\xi_j/\xi_0$ for $1\le j\le n$, $w=(w_1,\ldots,w_n)$ and $\lVert w \rVert^2:=\sum_{j=1}^n|w_j|^2$.

\subsection{Properties of CR maps between real hypersurfaces in complex Euclidean spaces}
For $j=1,2$, let $S_j\subset \mathbb C^{N_j}$ be a real hypersurface with a defining function $\rho_j$ satisfying $d\rho_j\neq 0$ on $S_j$, i.e., $S_j:=\{z\in \mathbb C^{N_j}: \rho_j(z)=0\}$.
Let $f: U \subset \mathbb C^{N_1} \to \mathbb C^{N_2}$ be a holomorphic map such that $f(S_1)\subset S_2$ with $S_1 \subset U$.
Then, there exists a smooth function $\vartheta$ on $U \subset \mathbb C^{N_1}$ such that
\begin{equation}\label{Eq:defining_functions}
\rho_2\circ f = \vartheta \cdot \rho_1.
\end{equation}
Fix a point $z_0\in S_1$. For any $v\in T^{1,0}_{z_0}(S_1)$ we have $f_*(v)\in T^{1,0}_{f(z_0)}(S_2)$.
Moreover, we have the following lemma.

\begin{lem}\label{lem:Preserve Levi form}
In the above settings, for any $v,v'\in T^{1,0}_{z_0}(S_1)$, $z_0\in S_1$, we have
\begin{equation}\label{Eq:Levi form}
 (\partial \overline\partial \rho_2)(f(z_0))(f_*(v),\overline{f_*(v')})
= \vartheta(z_0) (\partial \overline\partial\rho_1) (z_0)(v,\overline{v'}).
\end{equation}
\end{lem}
\begin{proof}
This is done by applying $\partial \overline\partial$ to both sides of (\ref{Eq:defining_functions}).
Applying $\partial \overline\partial$ to $\rho_2\circ f$, for any $v,v'\in T^{1,0}_{z_0}(S_1)$ we have
\[ (\partial \overline\partial(\rho_2\circ f))(z_0)(v,\overline{v'})
= f^*(\partial \overline\partial \rho_2) (z_0) (v,\overline{v'})
=(\partial \overline\partial \rho_2)(f(z_0))(f_*(v),\overline{f_*(v')}). \]
On the other hand, we apply $\partial \overline\partial$ to $\vartheta \cdot \rho_1$. Then, for any $v,v'\in T^{1,0}_{z_0}(S_1)$ we have
\[ (\partial \overline\partial(\vartheta \cdot \rho_1))(z_0)(v,\overline{v'})
= \vartheta(z_0) (\partial \overline\partial\rho_1) (z_0)(v,\overline{v'}) \]
because $\partial \rho_1(v) =0$, $\overline\partial \rho_1(\overline{v'})=0$ and $\rho_1(z_0)=0$.
Hence, we have obtained (\ref{Eq:Levi form}) as desired.
\end{proof}

\subsection{On the unit sphere bundle of $\mathbb P^m$}
\label{Sec:USB_Pm}
We let $m$ and $p$ be integers such that $1\le p \le m$.
Fixing a real constant $r>0$, we define the unit sphere bundle of $\mathbb P^m$ by
\[ S_r:=\left\{(\zeta,\xi)\in \wedge^p T_{\mathbb P^m}: \omega_{\mathbb P^m}^p(\zeta)(\xi,\overline\xi) = r \right\}, \]
where $\zeta=[\zeta_0,\zeta_1,\ldots,\zeta_m]\in \mathbb P^m$.
Write $\rho_r(\zeta,\xi):=\omega_{\mathbb P^m}^p(\zeta)(\xi,\overline\xi)-r$. 

We consider the point $x_0:=[1,0,\ldots,0]\in \mathbb P^m$ and we put $z=(z_1,\ldots,z_m)$ with $z_j:=\zeta_j/\zeta_0$ for $j=1,\ldots,m$.
Write $I=(i_1,\ldots,i_p)$ for the multi-index such that $1\le i_1<\cdots<i_p\le m$ and $\{s_I\}_{I}$ for the standard local basis of $\wedge^p T_{\mathbb P^m}$, where $s_I={\partial\over \partial z_{i_1}}\wedge \cdots \wedge {\partial\over \partial z_{i_p}}$.
Denote by ${\bf A}:=\{I=(i_1,\ldots,i_p): 1\le i_1<\cdots<i_p\le m\}$ the set of such multi-indices, which is a finite set with $|{\bf A}|={m\choose p}$.
Fixing $\xi_0\in \wedge^p T_{x_0}(\mathbb P^m)$, we may write $\xi_0 = \sum_{I\in {\bf A}} (\xi_0)_I s_I$.
Then, locally around the point $x_0=[1,0,\ldots,0]\in \mathbb P^m$ we may write $\xi = \sum_{I\in {\bf A}} \xi_I s_I$ and identify $\rho_r$ as a real-analytic function in $(z,\xi)$, namely
\[ \rho_r(z,\xi)
= C \sum_{I,J\in {\bf A}} \det( (g_{\mathbb P^m})_{i_s\overline{j_t}}(z) )_{1\le s,t\le p} \;\xi_I \overline{\xi_J} - r \]
for some real constant $C>0$ independent of the choice of local holomorphic coordinates, where $I=(i_1,\ldots,i_p)$, $J=(j_1,\ldots,j_p)\in {\bf A}$.
We also identify $x_0$ as the point ${\bf 0}\in \mathbb C^m$.

\medskip
\paragraph{\bf Holomorphic tangent space to $S_r$ at $(x_0,\xi_0)\in S_r$}
For $\mu=1,\ldots,m$ we have
${\partial \rho_r \over \partial z_\mu}({\bf 0},\xi_0) = 0$
since ${\partial\over \partial z_\mu}\det( (g_{\mathbb P^m})_{i_s\overline{j_t}}(z) )_{1\le s,t\le p}\big|_{z={\bf 0}}=0$.
For any $I\in {\bf A}$ we have ${\partial \rho_r \over \partial \xi_I}({\bf 0},\xi_0) = C \overline{(\xi_0)_I}$.
Thus, the holomorphic tangent space $T^{1,0}_{({\bf 0},\xi_0)}(S_r)$ to $S_r$ at $({\bf 0},\xi_0)$ consists of the vectors $\sum_{\mu=1}^m v_\mu {\partial \over \partial z_\mu} + \sum_{I\in {\bf A}} a_I s_I$ for $(v_1,\ldots,v_m)\in \mathbb C^m$ and $a_I\in \mathbb C$, $I\in {\bf A}$, with $\sum_{I\in {\bf A}} a_I \overline{(\xi_0)_I} = 0$.

\medskip
\paragraph{\bf Case $p=m$}
If $p=m$, then $|{\bf A}|=1$ and thus $a_{I}=0$, where $I=(1,\ldots,m)$.
In particular, the holomorphic tangent space $T^{1,0}_{(x_0,\xi_0)}(S_r)$ to $S_r$ at $(x_0,\xi_0)$ consists of the vectors $\sum_{\mu=1}^m v_\mu {\partial \over \partial z_\mu}$.

\medskip
\paragraph{\bf Case $p<m$}
Since $(x_0,\xi_0)\in S_r$, $(\xi_0)_{I_0} \neq 0$ for some $I_0\in {\bf A}$.
If $p<m$, then we may write
$a_{I_0} = -\sum_{I\in {\bf A},I\neq I_0} {a_I \overline{(\xi_0)_I}\over \overline{(\xi_0)_{I_0}}}$
and $T^{1,0}_{(x_0,\xi_0)}(S_r)$ consists of the vectors
\[ \sum_{\mu=1}^m v_\mu {\partial \over \partial z_\mu} + \sum_{I\in {\bf A}\smallsetminus \{I_0\}} a_I s_I -\sum_{I\in {\bf A}\smallsetminus \{I_0\}} {a_I \overline{(\xi_0)_I}\over \overline{(\xi_0)_{I_0}}} s_{I_0} \] 
for $(v_1,\ldots,v_m)\in \mathbb C^m$ and $a_I\in \mathbb C$, $I\in {\bf A}\smallsetminus \{I_0\}$.
Since ${\bf A}$ is a finite set, we may fix an ordering on ${\bf A}$ so that $I_0$ is the first element of ${\bf A}$.
Then, we define a row vector
${\bf b}:=\left(\ldots,-{\overline{(\xi_0)_I}\over \overline{(\xi_0)_{I_0}}},\ldots\right)_{I\in {\bf A}\smallsetminus \{I_0\}} \in M(1,|{\bf A}|-1;\mathbb C)$
and a column vector ${\bf a}:=(\ldots,a_I,\ldots)_{I\in {\bf A}\smallsetminus \{I_0\}}^t\in M(|{\bf A}|-1,1;\mathbb C)$.
Now, we have the transformation from $\mathbb C^{m+|{\bf A}|-1}\cong M(m+|{\bf A}|-1,1;\mathbb C)$ onto $T^{1,0}_{(x_0,\xi_0)}(S_r)$ in terms of matrix representations, as follows.
\[ \begin{bmatrix}v_1,\ldots, v_m, a_{I_0}, {\bf a}^t\end{bmatrix}^t
={\bf M} \begin{bmatrix} v_1,\ldots, v_m, {\bf a}^t \end{bmatrix}^t,
\text{ where } {\bf M}:=\begin{bmatrix}
{\bf I}_m & {\bf 0}\\
 {\bf 0}  & {\bf b}\\
 {\bf 0}  & {\bf I}_{|{\bf A}|-1}
\end{bmatrix}. \]

\medskip
\paragraph{\bf Computation of the Levi form}
Now, we compute $\partial\overline\partial \rho_r$ at $(x_0,\xi_0)\in S_r$.
Actually, it is similar to that in Yuan \cite{Yu17} for the case of the unit sphere bundle over $\mathbb B^m$.
The $(1,1)$-form $\partial\overline\partial \rho_r$ at $(x_0,\xi_0)\in S_r$ is represented by the complex Hessian matrix 
\[ {\bf H}:=\begin{bmatrix}
(\partial_{z_l}\partial_{\overline{z_k}}\rho_r ({\bf 0},\xi_0))_{1\le l,k\le m} & {\bf 0}\\
{\bf 0} & (\partial_{\xi_{I}}\partial_{\overline{\xi_{J}}}\rho_r ({\bf 0},\xi_0))_{I,J\in {\bf A}}
\end{bmatrix}\]
with
\[ \partial_{z_l}\partial_{\overline{z_k}}\rho_r ({\bf 0},\xi_0)
= - C \sum_{I,J\in {\bf A}} \Theta_{\wedge^p T_{\mathbb P^m}}
\left({\partial\over \partial z_l},{\partial\over \partial \overline{z_k}},s_I,\overline{s_J}\right)({\bf 0}) (\xi_0)_I \overline{(\xi_0)_J}
\]
and
\[ (\partial_{\xi_{I}}\partial_{\overline{\xi_{J}}}\rho_r ({\bf 0},\xi_0))
= C\left(\delta_{IJ}\right)_{I,J\in {\bf A}}. \]
Our goal here is to show that the Levi form $\partial \overline\partial \rho_r (\zeta,\xi)|_{T^{1,0}_{(\zeta,\xi)}(S_r)}$ at each point $(\zeta,\xi)\in S_r$ is non-degenerate.

\medskip
\paragraph{\bf Case $p=m$}
Suppose $p=m$. We have shown that $T^{1,0}_{(x_0,\xi_0)}(S_r)$ only consists of the vectors $\sum_{\mu=1}^m v_\mu {\partial \over \partial z_\mu}$.
Thus, the Levi form $\partial \overline\partial \rho_r(x_0,\xi_0)|_{T^{1,0}_{(x_0,\xi_0)}(S_r)}$ is represented by the Hermitian matrix
\[ (\partial_{z_l}\partial_{\overline{z_k}}\rho_r ({\bf 0},\xi_0))_{1\le l,k\le m}. \]
Since $(T_{\mathbb P^m},g_{\mathbb P^m})$ is Griffiths positive, the holomorphic line bundle $\wedge^m T_{\mathbb P^m}=\det(T_{\mathbb P^m})$ is positive with respect to the Hermitian metric induced from $g_{\mathbb P^m}$.
Then, the Hermitian inner product $H_m$ defined by
\[ H_m(\eta,\overline{\eta'}):=-\Theta_{\wedge^m T_{\mathbb P^m}}(\eta,\overline{\eta'},\xi_0,\overline{\xi_0})\quad \forall \;\eta,\eta'\in T^{1,0}_{x_0}(\mathbb P^m), \]
is negative definite. This shows that if $p=m$, then $\partial \overline\partial \rho_r(x_0,\xi_0)|_{T^{1,0}_{(x_0,\xi_0)}(S_r)}$ is negative definite, hence $S_r$ is Levi non-degenerate at $(x_0,\xi_0)$ of zero signature. By similar computations, when $p=m$, $S_r$ is Levi non-degenerate of zero signature and $\partial \overline\partial \rho_r|_{T^{1,0}(S_r)}$ is negative definite at any point $(\zeta,\xi)\in S_r$.

\medskip
\paragraph{\bf Case $p<m$}
If $p<m$, then $|{\bf A}|={m \choose p}>1$ and the Levi form $\partial \overline\partial \rho_r(x_0,\xi_0)|_{T^{1,0}_{(x_0,\xi_0)}(S_r)}$ at $(x_0,\xi_0)$ is represented by the Hermitian matrix
\[\overline{\bf M}^t {\bf H} {\bf M}=
\begin{bmatrix}
(\partial_{z_l}\partial_{\overline{z_k}}\rho_r ({\bf 0},\xi_0))_{1\le l,k\le m} & {\bf 0}\\
{\bf 0} & C(\overline{\bf v}^t{\bf v} + {\bf I}_{|{\bf A}|-1})
\end{bmatrix}.
\] 
The Hermitian matrix $C(\overline{\bf v}^t{\bf v} + {\bf I}_{|{\bf A}|-1})$ is positive definite, hence all its eigenvalues are positive.
We still need to consider the Hermitian inner product $H_p$ defined by
\[ H_p(\eta,\overline{\eta'}):=-\Theta_{\wedge^p T_{\mathbb P^m}}(\eta,\overline{\eta'},\xi_0,\overline{\xi_0})
\quad\forall\;\eta,\eta'\in T^{1,0}_{x_0}(\mathbb P^m). \]
To show that $H_p$ is negative definite, it suffices to show that $\wedge^p T_{\mathbb P^m}$ is Griffiths positive with respect to the Hermitian metric induced from $g_{\mathbb P^m}$.
By \cite[Equation (4.5'), p.\,258]{De12}, for any holomorphic tangent vector $\eta\in T^{1,0}_{x_0}(\mathbb P^m)$, $I=(i_1,\ldots,i_p)$, $J=(j_1,\ldots,j_p)\in {\bf A}$, we have
\[ \Theta_{\wedge^p T_{\mathbb P^m}}\left(\eta,\overline{\eta},s_I,\overline{s_J}\right)
= \sum_{k=1}^p \Theta_{T_{\mathbb P^m}}\left(\eta,\overline{\eta},{\partial\over \partial z_{i_k}},{\partial\over \partial \overline{z_{j_k}}}\right). \]
Since $(T_{\mathbb P^m},g_{\mathbb P^m})$ is Griffiths positive, so is $\wedge^p T_{\mathbb P^m}$ with respect to the Hermitian metric induced from $g_{\mathbb P^m}$.
Hence, $H_p$ is negative definite so that $H_p$ does not have zero eigenvalues and $H_p$ only has negative eigenvalues.
It follows that the Levi form $\partial \overline\partial \rho_r|_{T^{1,0}(S_r)}$ at $(x_0,\xi_0)$ has both negative and positive eigenvalues, and does not have zero eigenvalues. Hence, if $p<m$, then $S_r$ is Levi non-degenerate of positive signature $\le {1\over 2}\big(\dim_{\mathbb C} \wedge^p T_{\mathbb P^m} - 1\big)$. The same computation for $\partial \overline\partial \rho_r|_{T^{1,0}(S_r)}$ also works at any point $(\zeta,\xi)\in S_r$ so that $S_r$ is Levi non-degenerate of positive signature $\le {1\over 2}\left(\dim_{\mathbb C} \wedge^p T_{\mathbb P^m} - 1\right)$ when $p<m$.

\section{Local holomorphic maps between K\"ahler manifolds preserving $(p,p)$-forms}\label{Sec:Local holo map pp form}
Let $(X,g)$ be a complex $n$-dimensional K\"ahler manifold with the K\"ahler form $\omega$.
For simplicity, in the present article we write $(X,\omega)$ to indicate that $\omega$ is the K\"ahler form.
In terms of local holomorphic coordinates $(z_1,\ldots,z_n)$ on $X$, we may write $\omega=\sqrt{-1}\sum_{i,j=1}^n g_{i\overline j} dz_i\wedge d\overline{z_j}$ locally.
Then, for $1\le p\le n$ we have
\[ \omega^p = (\sqrt{-1})^p p! 
\sum_{\begin{tiny}\begin{split}
&1\le i_1<\cdots<i_p\le n,\\
&1\le j_1<\cdots<j_p\le n
\end{split}\end{tiny}}
 \det\begin{pmatrix} g_{i_s\overline{j_t}} \end{pmatrix}_{1\le s,t\le p}
 (dz_{i_1}\wedge d\overline{z_{j_1}}) \wedge \cdots \wedge 
 (dz_{i_p}\wedge d\overline{z_{j_p}}). \]

Let $(M,g_M)$ and $(N,g_N)$ be K\"ahler manifolds of complex dimensions $m$ and $n$ respectively.
Denote by $\omega_M$ and $\omega_N$ the K\"ahler forms of $(M,g_M)$ and $(N,g_N)$ respectively.
Fix some point $x_0\in M$.
Let $F:(M;x_0) \to (N;F(x_0))$ be a germ of holomorphic map such that
\begin{equation}\label{Eq:General_pp}
F^*\omega_N^p = \lambda \omega_M^p,
\end{equation}
for some integer $p$, $1\le p\le \min\{m,n\}$, and some positive real constant $\lambda$.
Such a germ of holomorphic map $F$ is clearly of rank at least $p$ near $x_0$.

Write $w=(w_1,\ldots,w_m)$ (resp.\,$z=(z_1,\ldots,z_n)$) for the local holomorphic coordinates on $M$ (resp.\,$N$) around the point $x_0$ (resp.\,$F(x_0)$). In what follows we write
\[ \omega_M = \sqrt{-1}\sum_{i,j=1}^m (g_M)_{i\overline j} dw_i\wedge d\overline{w_j},\quad \omega_N = \sqrt{-1} \sum_{i,j=1}^n (g_N)_{i\overline j} dz_i\wedge d\overline{z_j} \]
and $F(w) = (F_1(w),\ldots,F_n(w))$ locally around $x_0$ in terms of the local holomorphic coordinates.
Write
\[ {\partial(F_{i_1},\ldots,F_{i_p})\over \partial(w_{l_1},\ldots,w_{l_p})}
:= \begin{pmatrix}
{\partial F_{i_1}\over \partial w_{l_1}} & \cdots & {\partial F_{i_1}\over \partial w_{l_p}}\\
\vdots & \ddots & \vdots\\
{\partial F_{i_p}\over \partial w_{l_1}} & \cdots & {\partial F_{i_p}\over \partial w_{l_p}}
\end{pmatrix} \]
for $1\le i_1<\cdots<i_p\le n$ and $1\le l_1<\cdots <l_p\le m$.
Then, we have
\[ \begin{split}
F^*\omega_N^p
=& (\sqrt{-1})^p (-1)^{p(p-1)\over 2} p! 
\sum_{\begin{tiny}\begin{split}
&1\le i_1<\cdots<i_p\le n,\\
&1\le j_1<\cdots<j_p\le n\end{split}\end{tiny}}
\det\begin{pmatrix} (g_N)_{i_s\overline{j_t}}(F(w)) \end{pmatrix}_{1\le s,t\le p} \\
& \hspace{7.5cm}\cdot F^*( dz_{i_1}\wedge \cdots \wedge dz_{i_p}
\wedge d\overline{z_{j_1}}\wedge \cdots \wedge d\overline{z_{j_p}})\\
=& (\sqrt{-1})^p (-1)^{p(p-1)\over 2} p! 
\sum_{\begin{tiny}\begin{split}
&1\le l_1<\cdots<l_p\le m,\\
&1\le k_1<\cdots<k_p\le m \end{split}\end{tiny}}
\Theta_{l_1,\ldots,l_p;\overline{k_1},\ldots,\overline{k_p}}
 dw_{l_1}\wedge \cdots \wedge dw_{l_p}
\wedge d\overline{w_{k_1}}\wedge \cdots \wedge d\overline{w_{k_p}},\\
\end{split}\]
where
\[ \begin{split}
&\Theta_{l_1,\ldots,l_p;\overline{k_1},\ldots,\overline{k_p}}\\
:=&\sum_{\begin{tiny}\begin{split}
&1\le i_1<\cdots<i_p\le n,\\
&1\le j_1<\cdots<j_p\le n\end{split}\end{tiny}}
\det\begin{pmatrix} (g_N)_{i_s\overline{j_t}}(F(w)) \end{pmatrix}_{1\le s,t\le p} \det \left({\partial(F_{i_1},\ldots,F_{i_p})\over \partial(w_{l_1},\ldots,w_{l_p})}\right)\det \left(\overline{{\partial(F_{j_1},\ldots,F_{j_p})\over \partial(w_{k_1},\ldots,w_{k_p})}}\right). 
\end{split}\]
Then, in terms of the local holomorphic coordinates, (\ref{Eq:General_pp}) becomes
\begin{equation}\label{Eq:System_Eqs_pp}
\begin{split}
&\sum_{\begin{tiny}\begin{split}
&1\le i_1<\cdots<i_p\le n,\\
&1\le j_1<\cdots<j_p\le n\end{split}\end{tiny}}
\det\begin{pmatrix} (g_N)_{i_s\overline{j_t}}(F(w)) \end{pmatrix}_{1\le s,t\le p}
\det \left({\partial(F_{i_1},\ldots,F_{i_p})\over \partial(w_{l_1},\ldots,w_{l_p})}\right)
\det \left(\overline{{\partial(F_{j_1},\ldots,F_{j_p})\over \partial(w_{k_1},\ldots,w_{k_p})}}\right)\\
=& \lambda \det\begin{pmatrix} (g_M)_{l_s\overline{k_t}}(w) \end{pmatrix}_{1\le s,t\le p} 
\end{split}
\end{equation}
for $1\le l_1<\cdots<l_p \le m$ and $1\le k_1<\cdots<k_p\le m$.

\subsection{General situation}

Now, we observe the following rigidity result (i.e.,\,Proposition \ref{thm:pp_form_to_Holo_iso_0}) on local holomorphic maps from $M$ to $N$ between K\"ahler manifolds $M$ and $N$ preserving $(p,p)$-forms when $1\le p<\dim_{\mathbb C}(M)$.

\begin{pro}[Proposition \ref{thm:pp_form_to_Holo_iso_0}]\label{thm:pp_form_to_Holo_iso}
Let $(M,\omega_M)$ and $(N,\omega_N)$ be K\"ahler manifolds of finite complex dimensions $m$ and $n$ respectively, where $\omega_M$ and $\omega_N$ denote the corresponding K\"ahler forms.
Let $F:(M;x_0)\to (N;F(x_0))$ be a germ of holomorphic map such that $F^*\omega_N^p = \lambda \omega_M^p$ for some real constant $\lambda>0$ and some integer $p$, $1\le p\le \min\{m,n\}$. Then, $m\le n$.

If we assume in addition that $1\le p<m$, then $F^*\omega_N=\lambda^{1\over p} \omega_M$ so that $F:(M,\lambda^{1\over p} \omega_M;x_0)\to (N,\omega_N;F(x_0))$ is a local holomorphic isometry.
\end{pro}

\begin{proof}
If $p=m$, then $m \leq n$, since otherwise $\omega_N^p=0$. Then, it suffices to consider the case when $1\le p<m$.

Let $y_0$ be an arbitrary point in the domain of $F$.
Since $M$ is K\"ahler, we may choose local holomorphic coordinates $(z_1,\ldots,z_m)$ around $y_0$ on $M$ such that $\omega_M(y_0) = \sqrt{-1} \sum_{j=1}^m dz_j\wedge d\overline{z_j}$.
On the other hand, $(F^*\omega_N)(y_0)$ is represented by a Hermitian matrix $H$ that is positive semi-definite.

By identifying $y_0$ with ${\bf 0}\in \mathbb C^m$ and applying a suitable unitary transformation around ${\bf 0}$ in $\mathbb C^m$, we can diagonalize the Hermitian matrix $H$ and the Hermitian matrix representing the $(1,1)$-form $\omega_M(y_0)$ is still the identity matrix.
Actually, there is a unitary matrix ${\bf U}\in U(m)$ such that
\[ \overline{\bf U}^t H {\bf U} = \begin{pmatrix}
\lambda_1 & & \\
 & \ddots & \\
 & & \lambda_m
\end{pmatrix} \]
for some non-negative real numbers $\lambda_j$, $1\le j\le m$, and it is obvious that $\overline{\bf U}^t {\bf I}_m {\bf U} = {\bf I}_m$.

Now, we may still use $(z_1,\ldots,z_m)$ to denote the new local holomorphic coordinates because $\omega_M(y_0)$ remains unchanged under the transformation.
In particular, we can write $(F^*\omega_N)(y_0) = \sqrt{-1} \sum_{j=1}^m \lambda_j dz_j\wedge d\overline{z_j}$.
By the equation $F^*\omega_N^p = \lambda \omega_M^p$ at the point $y_0$, we have
\[ \left(\sum_{j=1}^m \lambda_j dz_j\wedge d\overline{z_j}\right)^p
= \lambda \left(\sum_{j=1}^m dz_j\wedge d\overline{z_j}\right)^p, \]
equivalently, 
\begin{equation}\label{Eq:Prod_EigenValue}
\prod_{k=1}^p \lambda_{j_k} = \lambda \text{ for } 1\le j_1<\cdots<j_p \le m.
\end{equation}
This implies that $\lambda_j\neq 0$ for $1\le j\le m$.
For any distinct $\mu,\nu$, $1\le \mu,\nu\le m$, since $p-1 \le m-2$ by the assumption $p<m$, we can choose distinct integers $i_1,\ldots,i_{p-1} \in \{1,\ldots,m\} \smallsetminus \{\mu,\nu\}$ such that by (\ref{Eq:Prod_EigenValue}) we have
\[ \lambda_\mu \prod_{k=1}^{p-1} \lambda_{i_k} 
= \lambda_\nu \prod_{k=1}^{p-1} \lambda_{i_k} = \lambda, \]
which implies $\lambda_\mu=\lambda_\nu$.
Therefore, we have $\lambda_1=\cdots=\lambda_m=\lambda^{1\over p}$, and thus $(F^*\omega_N)(y_0) = \lambda^{1\over p} \omega_M(y_0)$.
This identity indeed holds true for an arbitrary point $y_0$ in the domain of $F$.
Hence, $F^*\omega_N = \lambda^{1\over p} \omega_M$ on the domain of $F$, and thus $F:(M,\lambda^{1\over p} \omega_M;x_0)\to (N,\omega_N;F(x_0))$ is a local holomorphic isometry, as desired.
It follows readily that $F$ is of rank $m$ and thus $m\le n$.
\end{proof}

Proposition \ref{thm:pp_form_to_Holo_iso} reduces the whole study on local holomorphic maps between K\"ahler manifolds $M$ and $N$ preserving $(p,p)$-forms that are induced from the K\"ahler forms of $M$ and $N$, into the following two categories.
\begin{enumerate}
\item The study of local holomorphic isometries between K\"ahler manifolds $M$ and $N$.
\item The study of local holomorphic maps $F:(M;x_0)\to (N;F(x_0))$ such that $F^*\omega_N^m = \lambda \omega_M^m$ for some real constant $\lambda>0$, where $M$ and $N$ are K\"ahler manifolds of complex dimensions $m$ and $n$ respectively, and $m\le n$.
\end{enumerate}

Now, we also complete the proof of Theorem \ref{thm:local holo map_BSD_pp_form}, as follows.

\begin{proof}[Proof of Theorem \ref{thm:local holo map_BSD_pp_form}]
We first suppose $1\le p < n $.
Then, it follows from Proposition \ref{thm:pp_form_to_Holo_iso_0} that $f:(D,\lambda^{1\over p} ds_D^2;x_0) \to (\Omega,ds_\Omega^2;y_0)$ is a germ of holomorphic isometry and in particular, $n \leq N$.

Now, we suppose $n=N=p$.
Write $f=(f_1,\ldots,f_n)$, $w=(w_1,\ldots,w_n)\in D \Subset \mathbb C^n$, $z=(z_1,\ldots,z_n)\in \Omega \Subset \mathbb C^n$, and 
\[ \omega_D=\sqrt{-1}\sum_{i,j=1}^n (\omega_D)_{i\overline j}  dw_i\wedge d\overline{w_j},\qquad \omega_\Omega=\sqrt{-1}\sum_{i,j=1}^n (\omega_\Omega)_{i\overline j}  dz_i\wedge d\overline{z_j}. \]
Denote by $Jf(w):=\begin{pmatrix} {\partial f_i\over \partial w_j}(w) \end{pmatrix}_{1\le i,j\le n}$ the Jacobian matrix of $f$.
Then, $f^*\omega_{\Omega}^p = \lambda \omega_D^p$ is equivalent to
\[ \det\begin{pmatrix} (\omega_\Omega)_{i\overline j} (f(w)) \end{pmatrix}_{1\le i,j\le n} |\det Jf(w)|^2
= \lambda \det\begin{pmatrix} (\omega_D)_{i\overline j}(w) \end{pmatrix}_{1\le i,j\le n}. \]
Therefore, we have
\[ -\sqrt{-1}\partial\overline\partial \log \det\begin{pmatrix} (\omega_\Omega)_{i\overline j} (f(w)) \end{pmatrix}_{1\le i,j\le n}
= -\sqrt{-1} \partial\overline\partial \log \det\begin{pmatrix} (\omega_D)_{i\overline j}(w) \end{pmatrix}_{1\le i,j\le n}, \]
i.e., $f^* {\rm Ric}(\Omega,ds_\Omega^2)={\rm Ric}(D,ds_D^2)$.
It is well known that for any bounded symmetric domain $U$ we have ${\rm Ric}(U,ds_U^2)=-\omega_{U}$ (cf.\,Mok \cite[p.\,59]{Mo89}).
Hence, we have $f^*\omega_{\Omega} = \omega_D$, i.e., $f:(D,ds_D^2;x_0)\to (\Omega,ds_\Omega^2;y_0)$ is a local holomorphic isometry, and thus $\lambda=1$ by the identity $f^*\omega_{\Omega}^p = \lambda \omega_D^p$.

In both cases, it follows from the extension theorem of Mok \cite[Theorem 1.3.1, p.\,1618]{Mo12} that $f$ extends to a proper holomorphic isometric embedding $F:(D,\lambda^{1\over p} ds_D^2) \to (\Omega,ds_\Omega^2)$.
When $p=n=N$, we have shown that $\lambda=1$ in the above.
If we assume in addition that each irreducible factor of $D$ is of rank $\ge 2$, then $F$ is totally geodesic by \cite[Theorem 1.3.2, p.\,1636]{Mo12}, and so is $f$.
The last assertion follows directly.
\end{proof}

\subsection{Technique of using the unit sphere bundles}
\label{Sec:Tech_USB}
Let $(M,\omega_M)$ and $(N,\omega_N)$ be K\"ahler manifolds of complex dimensions $m$ and $n$ respectively.
Let $F:(M;x_0) \to (N;F(x_0))$ be a germ of holomorphic map such that $F^*\omega_N^p = \lambda \omega_M^p$ for some integer $p$, $1\le p\le \min\{m,n\}$, and some real constant $\lambda>0$.
Write $U$ for an open neighborhood of $x_0$ in the domain of $F$ in which we have the holomorphic coordinates $(w_1,\ldots,w_m)$ on $U$.
Define the unit sphere bundles
\begin{equation}\label{Eq:USphereB1}
S_1:=\left\{(w,v)\in \wedge^p T_{U}: \lambda \omega_M^p(w)(v,\overline v) = 1\right\},\qquad S_2:=\left\{(z,\xi)\in \wedge^p T_N: \omega_N^p(z)(\xi,\overline\xi) = 1\right\}.
\end{equation}
The map $F$ induces the holomorphic map $f:=(F,dF): \wedge^p T_{U}\to \wedge^p T_N$ defined by $f(w,v):= (F(w),dF_w(v))$ for all $(w,v)\in \wedge^p T_{U}$.
By (\ref{Eq:General_pp}), $f$ maps $S_1$ into $S_2$. Writing $\rho_1(w,v):=\lambda \omega_M^p(w)(v,\overline v) - 1$ and $\rho_2(z,\xi):=\omega_N^p(z)(\xi,\overline\xi) - 1$, we also define
\[ U_1:=\left\{(w,v)\in \wedge^p T_{U}: \rho_1(w,v)< 0 \right\},\qquad 
V_1:=\left\{(w,v)\in \wedge^p T_{U}: \rho_1(w,v) > 0 \right\}
 \]
and
\[ U_2:=\left\{(z,\xi)\in \wedge^p T_N: \rho_2(z,\xi) < 0 \right\},\qquad
 V_2:=\left\{(z,\xi)\in \wedge^p T_N: \rho_2(z,\xi) > 0 \right\}.
\]
Then, by (\ref{Eq:General_pp}) and the definitions of $\rho_1$ and $\rho_2$, we have $\rho_2\circ f= \rho_1$ so that $f=(F,dF)$ maps $U_1$ into $U_2$, and maps $V_1$ into $V_2$.

As the continuation of the study in Yuan \cite{Yu17}, one may make use of the method of CR geometry to study the CR map $f:S_1\to S_2$.
To establish a non-existence theorem for such a germ of holomorphic map $F$ between certain K\"ahler manifolds $M$ and $N$, one key idea is to show the non-existence of a CR map between the real hypersurfaces $S_1$ and $S_2$ induced from the K\"ahler manifolds $(M,\omega_M)$ and $(N,\omega_N)$. In particular, we have the following general theorem.

\begin{thm}\label{thm:General_M_N_curva assump}
Let $(M,\omega_M)$ and $(N,\omega_N)$ be K\"ahler manifolds of complex dimensions $m$ and $n$ respectively such that $(M,\omega_M)$ is of semipositive holomorphic bisectional curvature and $(N,\omega_N)$ is of negative holomorphic bisectional curvature.
Then, there does not exist any germ of holomorphic map $F:(M;x_0)\to (N;F(x_0))$ such that $F^*\omega_N^p=\lambda \omega_M^p$ holds, where $\lambda>0$ is a real constant and $p$ is an integer with $1\le p\le \min\{m,n\}$.
\end{thm}
\begin{proof}
Assume the contrary that there exists a germ of holomorphic map $F:(M;x_0) \to (N;F(x_0))$ such that $F^*\omega_N^p = \lambda \omega_M^p$ holds.
Write $U\subset M$ for the domain of $F$.
Then, in the above settings we have the CR map $f:=(F,dF):S_1\to S_2$ between the unit sphere bundles $S_1$ and $S_2$ defined in (\ref{Eq:USphereB1}), and we write $\rho_1$ and $\rho_2$ for the defining functions of $S_1$ and $S_2$ respectively as above.
We have $\rho_2\circ f = \rho_1$, by $F^*\omega_N^p=\lambda \omega_M^p$.
Fix a point $(w_0,v_0)\in S_1$.
By Lemma \ref{lem:Preserve Levi form}, for any $\eta,\eta'\in T^{1,0}_{(w_0,v_0)}(S_1)$, we have
\begin{equation}\label{Eq:Levi form_General_M_N_curva assump}
 (\partial \overline\partial \rho_2)(f(w_0,v_0))(f_*(\eta),\overline{f_*(\eta')}) 
= (\partial \overline\partial\rho_1) (w_0,v_0)(\eta,\overline{\eta'}).
\end{equation}

By similar computations in Section \ref{Sec:USB_Pm} (cf.\,Yuan \cite{Yu17}), since $(N,\omega_N)$ is of negative holomorphic bisectional curvature, $S_2$ is strictly pseudoconvex.
By the identity $F^*\omega_N^p = \lambda \omega_M^p$, $F$ is of rank $\ge p$ at every point in $U$. 
Thus, we may choose a vector $\eta=(\eta_1,{\bf 0})\in T^{1,0}_{(w_0,v_0)}(S_1)$ with $\eta_1\in T_{w_0}(U)$ and $f_*(\eta)=(dF_{w_0}(\eta_1),\eta')\neq {\bf 0}$.
Since $(M,\omega_M)$ is of semipositive holomorphic bisectional curvature, the holomorphic tangent bundle $T_U$ is Griffiths semipositive, and thus $\wedge^p T_U$ is Griffiths semipositive with respect to the Hermitian metric induced from $\omega_M$ by similar arguments in Section \ref{Sec:USB_Pm}.
Therefore, by similar computations in Section \ref{Sec:USB_Pm}, we have
\[ (\partial\overline\partial \rho_1) (w_0,v_0)(\eta,\overline{\eta})
= -C\Theta_{\wedge^p T_U}(\eta_1,\overline{\eta_1},v_0,\overline{v_0}) \le 0 \] for some real constant $C>0$.
On the other hand, $(\partial \overline\partial \rho_2)(f(w_0,v_0))(f_*(\eta),\overline{f_*(\eta)})>0$ since $S_2$ is strictly pseudoconvex and $f_*(\eta)\neq {\bf 0}$.
This clearly contradicts with (\ref{Eq:Levi form_General_M_N_curva assump}), and the proof is complete.
\end{proof}

\begin{cor}
Let $(M,\omega_M)$ be a K\"ahler manifold of complex dimension $m$ and of semipositive holomorphic bisectional curvature. Then, there does not exist any germ of holomorphic map $F:(M;x_0)\to (\mathbb{B}^n; F(x_0))$ such that $F^*\omega_{\mathbb{B}^n}^p=\lambda \omega_M^p$ holds, where $\lambda>0$ is a real constant and $p$ is an integer with $1\le p\le \min\{m,n\}$.
\end{cor}

\subsection{Applications of the Umehara algebra}
By making use of the Umehara algebra, we have the following non-existence theorem.

\begin{thm}\label{thm:M_to_Cn_pp form}
Let $(M,\omega_M)$ be a K\"ahler manifold of complex dimension $m$, and we fix some point $x_0\in M$.
Let $F:(M;x_0) \to (\mathbb C^n;F(x_0))$ be a germ of holomorphic map such that
\begin{equation}\label{Eq:M_to_C_pp}
F^*\omega_{\mathbb C^n}^p = \lambda \omega_{M}^p
\end{equation}
for some integer $p$, $1\le p\le \min\{m,n\}$, and some real constant $\lambda>0$.
Then, $M$ can neither be $\mathbb B^m$ nor $\mathbb P^m$ equipped with the standard K\"ahler metric of constant holomorphic sectional curvature.
\end{thm}

Note that (\ref{Eq:M_to_C_pp}) is equivalent to the following system of equations
\begin{equation}\label{Eq:Sys_Eq_M_to_C_pp}
\sum_{1\le i_1<\cdots<i_p\le n}
\det \left({\partial(F_{i_1},\ldots,F_{i_p})\over \partial(w_{l_1},\ldots,w_{l_p})}\right)
\det\left( \overline{{\partial(F_{i_1},\ldots,F_{i_p})\over \partial(w_{k_1},\ldots,w_{k_p})}}\right)
= \lambda \det \left( (g_M)_{l_s\overline{k_t}}(w) \right)_{1\le s,t\le p}
\end{equation}
for $1\le l_1<\cdots<l_p \le m$ and $1\le k_1<\cdots<k_p\le m$.

\medskip
To prove the theorem, we need the following basic facts on Umehara algebra.
\begin{lem}\label{lem:det_Bm_not in UmAlg}
For $1 \le p \le m$, the real-analytic function
\[ \phi_p(w):={1 \over (1-\lVert w \rVert^2)^{2p}}\det\left( (1-\lVert w \rVert^2)\delta_{st}+\overline{w_{s}} w_{t}\right)_{1\le s,t\le p} \]
on $\mathbb B^m$ does not lie inside the Umehara algebra $\Lambda(\mathbb B^m)$.
\end{lem}
\begin{proof}
Assume the contrary that $\phi_p\in \Lambda(\mathbb B^m)$.
If $p=m$, then by the fact that $(\mathbb B^m,\omega_{\mathbb B^m})$ is K\"ahler-Einstein, we have
$\phi_m(w) = {1 \over (1-\lVert w \rVert^2)^{m+1}}$,
which does not lie inside $\Lambda(\mathbb B^m)$ by Umehara \cite[p.\,520]{Um88}.

In general, by restricting to $\mathbb B^m\cap \{(w_1,\ldots,w_m)\in \mathbb C^m: w_j=0\text{ for } j=2,\ldots,m\}$ we obtain a function
$\widetilde\phi_p(\zeta):=\phi_p(\zeta,0,\ldots,0)$
for $\zeta\in \{w_1\in \mathbb C: (w_1,0\ldots,0)\in \mathbb B^m\} = \{w_1\in \mathbb C:|w_1|^2<1\}=:\Delta$.
Then, we would have $\widetilde\phi_p\in \Lambda(\Delta)$ by the assumption.
But then we compute
\[ \widetilde\phi_p(\zeta)={1 \over (1-|\zeta |^2)^{2p}}
\det\begin{bmatrix}
(1-|\zeta|^2)+|\zeta|^2 & {\bf 0} \\
{\bf 0}  & (1-|\zeta|^2) {\bf I}_{p-1}
\end{bmatrix}
={1\over (1-|\zeta |^2)^{p+1}}, \]
which does not lie in $\Lambda(\Delta)$ by \cite{Um88}, a plain contradiction.
Hence, $\phi_p\not\in \Lambda(\mathbb B^m)$.
\end{proof}

\begin{lem}\label{lem:det_Pm_not in UmAlg}
Let $U$ be a simply connected open neighborhood of $x_0=[1,0,\ldots,0]$ such that $U\subset U_0 \subset \mathbb P^m$ and $\sum_{j=1}^m |\xi_j|^2 < |\xi_0|^2$ for all $[\xi_0,\xi_1,\ldots,\xi_m]\in U$.
We identify $U_0\cong \mathbb C^m$ via the map $U_0\ni [\xi_0,\xi_1,\ldots,\xi_m]\mapsto (w_1,\ldots,w_m)\in \mathbb C^m$, where $w_j:=\xi_j/\xi_0$ for $j=1,\ldots,m$.
In particular, $U$ is identified as an open neighborhood of ${\bf 0}$ in $\mathbb B^m$.
Then, for $1\le p\le m$ the function
\[ \Upsilon_p(w):={1 \over (1+\lVert w \rVert^2)^{2p}}\det\left( (1+\lVert w \rVert^2)\delta_{s t}-\overline{w_{s}} w_{t}\right)_{1\le s,t\le p} \]
does not lie in $\Lambda(U)$.
\end{lem}
\begin{proof}
We first consider the case where $m=1$. Then, $p=1$ and we have
$\Upsilon_1(w) = {1\over (1+| w|^2)^2}$.
Locally around $w=0$, the function $\Upsilon_1$ can be written as an infinite sum 
$\sum_{k=1}^{+\infty} |\phi_k(w)|^2 - \sum_{l=1}^{+\infty} |\psi_l(w)|^2$
for some linearly independent polynomials $\phi_k(w)$, $k\ge 1$ and $\psi_l(w)$, $l\ge 1$, in $w$.
By the definition of the ranks of real-analytic functions on complex manifolds in \cite{Um88}, $\Upsilon_1$ is not of finite rank so that $\Upsilon_1\not\in \Lambda(U)$.

Now, we suppose $m\ge 2$ and we fix $p$, $1\le p\le m$. Assume the contrary that $\Upsilon_p\in \Lambda(U)$.
By restricting to $U \cap \{(w_1,\ldots,w_m)\in \mathbb C^m: w_j=0\text{ for } j=2,\ldots,m \}$, we consider the function
$\widetilde\Upsilon_p(\zeta):=\Upsilon_p(\zeta,0,\ldots,0)$
for $\zeta\in \widetilde U$, where $\widetilde U$ is a simply connected open neighborhood of ${\bf 0}$ in $\{w_1\in \mathbb C: (w_1,0,\ldots,0)\in U\}$. Then, $\widetilde\Upsilon_p\in \Lambda(\widetilde U)$.
We compute
\[ \widetilde\Upsilon_p(\zeta)={1 \over (1+|\zeta |^2)^{2p}}
\det\begin{bmatrix}
(1+|\zeta|^2)-|\zeta|^2 & {\bf 0} \\
{\bf 0}  & (1+|\zeta|^2){\bf I}_{p-1}
\end{bmatrix}
={1\over (1+|\zeta |^2)^{p+1}}. \]
But then by the same arguments in the case of $m=1$, $\widetilde\Upsilon_p(\zeta)={1\over (1+|\zeta |^2)^{p+1}}$ is not of finite rank, and thus it does not lie in $\Lambda(\widetilde U)$ by \cite{Um88}, a plain contradiction. Hence, $\Upsilon_p\not\in \Lambda(U)$.
\end{proof}

\begin{proof}[Proof of Theorem \ref{thm:M_to_Cn_pp form}]
Assume the contrary that $M$ is either $\mathbb B^m$ or $\mathbb P^m$ equipped with the standard K\"ahler form $\omega_{M}$ of constant holomorphic sectional curvature defined in Section \ref{Sec:CSF}.
Let $U$ be a simply connected open neighborhood of $x_0$ in the domain of $F$.
By Lemmas \ref{lem:det_Bm_not in UmAlg} and \ref{lem:det_Pm_not in UmAlg}, we have $\det \left( (g_M)_{s\overline{t}}(w) \right)_{1\le s,t\le p} \not \in \Lambda(U)$.
But then from (\ref{Eq:Sys_Eq_M_to_C_pp}) we have
\[ \sum_{1\le i_1<\cdots<i_p\le n}
\left|\det \left({\partial(F_{i_1},\ldots,F_{i_p})\over \partial(w_{1},\ldots,w_{p})}\right)\right|^2
= \lambda \det \left( (g_M)_{s\overline{t}}(w) \right)_{1\le s,t\le p}, \]
which implies that $\det \left( (g_M)_{s\overline{t}}(w) \right)_{1\le s,t\le p}\in \Lambda(U)$, a plain contradiction. The proof is complete.
\end{proof}

In the consideration of local holomorphic maps from $\mathbb B^m$ into $\mathbb P^n$ preserving invariant $(p,p)$-forms, we also have the following non-existence result by using the Umehara algebra.

\begin{pro}\label{pro:Bm_Pn_pp form}
Let $n,m$ and $p$ be positive integers such that $1\le p\le \min\{n,m\}$.
Then, there does not exist a germ of holomorphic map $F:(\mathbb B^m;x_0)\to (\mathbb P^n;F(x_0))$ such that 
\begin{equation}\label{Eq:B_to_P_pp form}
F^*\omega_{\mathbb P^n}^p  = \lambda \omega_{\mathbb B^m}^p
\end{equation}
for some real constant $\lambda>0$.
\end{pro}
\begin{proof}
Assume the contrary that there exists such a germ of holomorphic map $F$ such that (\ref{Eq:B_to_P_pp form}) holds.
We may assume $x_0={\bf 0}$ and $F(x_0)=[1,0,\ldots,0]$ without loss of generality by composing $F$ with holomorphic isometries of $\mathbb B^m$ (resp.\,$\mathbb P^n$).
Let $U\subset \mathbb B^m$ be the domain of $F$.

Let ${\bf A}$ be the set of all multi-indices $I=(i_1,\ldots,i_p)$ such that $1\le i_1<\cdots<i_p\le n$.
We identify the map $F$ as $w\mapsto [1,F_1(w),\ldots,F_n(w)]\in \mathbb P^n$ for $w$ sufficiently close to $x_0={\bf 0}\in \mathbb B^m$, and write $\lVert F(w) \rVert^2:=\sum_{j=1}^n|F_j(w)|^2$.
Then, (\ref{Eq:B_to_P_pp form}) is equivalent to the following system of equations
\[ \begin{split}
&\sum_{I,J\in {\bf A}}\det( (g_{\mathbb P^n})_{i_s\overline{j_t}}(F(w)) )_{1\le s,t\le p}
\det\left({\partial (F_{i_1},\ldots,F_{i_p})\over \partial(w_{l_1},\ldots,w_{l_p})}\right)
\det\left(\overline{{\partial (F_{j_1},\ldots,F_{j_p})\over \partial(w_{k_1},\ldots,w_{k_p})}}\right)\\
=& \lambda \det \left( (g_{\mathbb B^m})_{l_s\overline{k_t}}(w) \right)_{1\le s,t\le p} 
\end{split}\]
for $1\le l_1<\cdots<l_p\le m$ and $1\le k_1<\cdots<k_p\le m$, where $I=(i_1,\ldots,i_p)$ and $J=(j_1,\ldots,j_p)$.
In what follows we put $l_s=k_s=s$ for $1\le s\le p$.
Define $\phi_p(w):=\det \left( (g_{\mathbb B^m})_{s\overline{t}}(w) \right)_{1\le s,t\le p}$.
Let $U'$ be a simply connected open neighborhood of ${\bf 0}$ in $U$.
Note that $\det( (g_{\mathbb P^n})_{i_s\overline{j_t}}(F(w)) )_{1\le s,t\le p}
= {1 \over (1+\lVert F(w) \rVert^2)^{2p}}
\det\left( M_{I\overline J}(F(w)) \right)$, where $M_{I\overline J}(F(w)):=\left( (1+\lVert F(w) \rVert^2)\delta_{i_s j_t} - \overline{F_{i_s}(w)}F_{j_t}(w) \right)_{1\le s,t\le p}$ for $I=(i_1,\ldots,i_p),J=(j_1,\ldots,j_p)\in {\bf A}$.
Moreover, we have $\overline{\det\left( M_{I\overline J}(F(w)) \right)}
=\det\left( M_{J\overline I}(F(w)) \right)$ for all $I,J\in {\bf A}$.
Therefore,
\[ (1+\lVert F(w) \rVert^2)^{2p}\phi_p(w)
={1\over \lambda}\sum_{I,J\in {\bf A}} \det \left( M_{I\overline J}(F(w)) \right)
\det\left({\partial (F_{i_1},\ldots,F_{i_p})\over \partial(w_1,\ldots,w_p)}\right)
\det\left(\overline{{\partial (F_{j_1},\ldots,F_{j_p})\over \partial(w_1,\ldots,w_p)}}\right)\]
so that $(1+\lVert F(w) \rVert^2)^{2p}\phi_p(w)\in \Lambda(U')$. 
If $m=1$, then $p=1$ and $\phi_1(w)={1\over (1-|w|^2)^2}$.
If $m\ge 2$, then we restrict to a simply connected open neigborhood $U''$ of ${\bf 0}$ in $\{\zeta\in \mathbb C: (\zeta,{\bf 0})\in U'\}$.
Defining the function $\psi$ on $U''$ by $\psi(\zeta):=(1+\lVert F(\zeta,{\bf 0}) \rVert^2)^{2p}\phi_p(\zeta,{\bf 0})$ for $\zeta\in U''$, we have $\psi\in \Lambda(U'')$ since $(1+\lVert F(w) \rVert^2)^{2p}\phi_p(w)\in \Lambda(U')$.
From the proof of Lemma \ref{lem:det_Bm_not in UmAlg} we have
$\phi_p(\zeta,{\bf 0})={1\over (1-|\zeta|^2)^{p+1}}$ and thus
\[ \psi(\zeta) = {(1+\lVert F(\zeta,{\bf 0}) \rVert^2)^{2p}\over (1-|\zeta|^2)^{p+1}}. \]
(Noting that this includes the case where $m=p=1$.)
On the other hand, by Umehara \cite[p.\,535]{Um88} we have $(1-|\zeta|^2)^{-(p+1)}=\sum_{j=1}^{+\infty}|h_j(\zeta)|^2$ for some linearly independent holomorphic functions $h_j$, $j=1,2,3,\ldots$, around $\zeta={\bf 0}$. This implies that
\[ \psi(\zeta)
= \sum_{j=1}^{+\infty}|h_j(\zeta)|^2 + \sum_{j=1}^{+\infty}|\widetilde h_j(\zeta)|^2 \]
for some holomorphic functions $\widetilde h_j$ around $\zeta={\bf 0}$, so that $\psi(\zeta)$ must be of infinite rank, i.e., $\psi\not \in \Lambda(U'')$ by \cite[Corollary 3.3, p.\,534]{Um88}, a plain contradiction. Hence, there does not exist such a germ of holomorphic map $F$, as desired.
\end{proof}

\subsection{Complete proof of Theorem \ref{thm:CSF diff type_non-exist_holo map_pp forms}}\label{Sec:Complete_proof_main thm}
We can now complete the proof of Theorem \ref{thm:CSF diff type_non-exist_holo map_pp forms}, as follows.

\begin{proof}[Proof of Theorem \ref{thm:CSF diff type_non-exist_holo map_pp forms}]
We may suppose $M$ and $N$ are simply connected complex space forms of different types as stated in the introduction of the present article.
\begin{enumerate}
\item If $N=\mathbb C^n$, then the result follows from Theorem \ref{thm:M_to_Cn_pp form}.
\item If $N=\mathbb B^n$, then the result follows from Theorem \ref{thm:General_M_N_curva assump} because $(\mathbb B^n,\omega_{\mathbb B^n})$ is of negative holomorphic bisectional curvature while $(M,\omega_M)$ is of semipositive holomorphic bisectional curvature if $M$ is either $\mathbb C^m$ or $\mathbb P^m$.
\item It remains to consider the case where $N=\mathbb P^n$. If $M=\mathbb B^m$, then the result follows from Propositions \ref{pro:Bm_Pn_pp form}. 
Now, suppose $M=\mathbb C^m$.
Assume the contrary that there exists such a local holomorphic map $F$ as stated in Theorem \ref{thm:CSF diff type_non-exist_holo map_pp forms}.
Then, by $F^*\omega_{\mathbb P^n}^p=\lambda \omega_{\mathbb C^m}^p$, $\lambda>0$, $F$ is not a constant map.
If $1\le p<m$, then by Proposition \ref{thm:pp_form_to_Holo_iso_0}, $F$ would be a local holomorphic isometry from $(\mathbb C^m, \lambda^{1\over p}\omega_{\mathbb C^m})$ into $(\mathbb P^n,\omega_{\mathbb P^n})$, which is impossible by results of Calabi \cite{Ca53}.
If $p=m=n\ge 2$, then since ${\rm Ric}(\mathbb P^n,\omega_{\mathbb P^n})=\omega_{\mathbb P^n}$ and ${\rm Ric}(\mathbb C^m,\omega_{\mathbb C^m})=0$, by $F^*\omega_{\mathbb P^n}^n = \lambda \omega_{\mathbb C^n}^n$ and similar arguments in the proof of Theorem \ref{thm:local holo map_BSD_pp_form}, we have $F^*\omega_{\mathbb P^n}=0$ so that $F$ is a constant map, a plain contradiction.
\end{enumerate}
The proof is complete.
\end{proof}

By the similar arguments in the proofs of Theorems \ref{thm:General_M_N_curva assump} and \ref{thm:M_to_Cn_pp form}, we may obtain the non-existence of local holomorphic maps from complex space forms into the product of complex space forms of different types preserving $(p, p)$-forms in the following.

\begin{pro}
Let $(M,\omega_M)$ be a complex space form, and $N:=N_1\times \cdots \times N_k$ be a product of complex space forms $(N_j,\omega_{N_j})$, $1\le j\le k$, such that each $N_j$ is not biholomorphic to a complex projective space, and $M$ and $N_j$ are of different types for $1\le j\le k$.
Write $m:=\dim_{\mathbb C}(M)$ and $n_j:=\dim_{\mathbb C}(N_j)$ for $1\le j\le k$.
Then, there does not exist a germ of holomorphic map $F:(M;x_0) \to (N;F(x_0))$ with $F=(F_1,\ldots,F_k)$ and $F_j:(M;x_0)\to (N_j;y_0^j)$, $1\le j\le k$, such that
\[ \sum_{j=1}^k \lambda_j F_j^* \omega_{N_j}^p = \omega_M^p \]
for some real constants $\lambda_j>0$, and some integer $p$, $1\le p\le \min\{m,n_1,\ldots,n_k\}$.
\end{pro}

\begin{rem}
It is still not clear whether there exists a germ of holomorphic map $F:(\mathbb C^m; x_0) \to (\mathbb P^n; F(x_0))$ such that
\begin{equation}\label{Eq:Cm_to_Pn_mm form}
F^*\omega_{\mathbb P^n}^m = \lambda \omega_{\mathbb C^m}^m
\end{equation}
for some real constant $\lambda>0$. In fact, {\rm(\ref{Eq:Cm_to_Pn_mm form})\rm} holds if and only if there exists a complex $m$-dimensional Ricci flat K\"ahler submanifold of $\mathbb{P}^n$. By  
{\rm(\ref{Eq:Cm_to_Pn_mm form})\rm}, the Ricci form of $(F(U),\omega_{\mathbb P^n}|_{F(U)})$ is zero by
\[ F^*{\rm Ric}(F(U),\omega_{\mathbb P^n}|_{F(U)})
= {\rm Ric}(U,F^*\omega_{\mathbb P^n}|_U) = {\rm Ric}(U,\omega_{\mathbb C^m}|_U)=0. \]
Thus, $(F(U),\omega_{\mathbb P^n}|_{F(U)})\subset (\mathbb P^n,\omega_{\mathbb P^n})$ is a K\"ahler-Einstein submanifold with zero Einstein constant. On the other hand, if there exists a complex $m$-dimensional Ricci flat K\"ahler submanifold of $\mathbb{P}^n$, by the standard reduction in \cite{Hu00}, {\rm(\ref{Eq:Cm_to_Pn_mm form})\rm} holds on an open subset $U \subset \mathbb C^m$.

The problem on the existence of a Ricci-flat K\"ahler submanifold of finite dimensional complex projective spaces is still open.
Moreover, there is a more general open problem regarding the classification of complex $m$-dimensional K\"ahler-Einstein submanifolds of finite complex $n$-dimensional complex space forms {\rm(}see \cite[Conjecture 1, p.\,7844]{LSZ21}{\rm)}. If the codimension $n-m$ is at most two, then this problem has been solved by Chern {\rm(}1967{\rm)} and Tsukada {\rm(}1986{\rm)}, and it follows that there does not exist a complex $m$-dimensional Ricci-flat K\"ahler submanifold of $\mathbb P^n$ for $n-m\le 2$ and $m\ge 2$ {\rm(}cf.\,\cite[Theorem 4.4, p.\,58]{LZ18}{\rm)}. The general situation of this problem is still unsolved.
\end{rem}

\section{Further discussions}
\subsection{Indefinite complex space forms}
A connected complete indefinite K\"ahler manifold $X$ is said to be an indefinite complex space form if $X$ is of constant holomorphic sectional curvature (see \cite[p.\,56]{BR82}).
We also write $X=X(c)$ to indicate that $X$ is of constant holomorphic sectional curvature $c$, where $c\in \mathbb R$.

Let $s$ and $n\ge 1$ be integers such that $0\le s \le n$.
For $w=(w_1,\ldots,w_n)\in \mathbb C^n$, we write $\lVert w \rVert_s^2:=\sum_{j=1}^s |w_j|^2 - \sum_{j=s+1}^n |w_j|^2$ for $1\le s \le n-1$, $\lVert w \rVert_0^2:= - \sum_{j=1}^n |w_j|^2$ and $\lVert w \rVert_n^2:=\sum_{j=1}^n |w_j|^2$.
We equip $\mathbb C^n$ with the indefinite K\"ahler form $\omega_{\mathbb C^n,s}$ given by
$\omega_{\mathbb C^n,s}:=\sqrt{-1}\partial \overline\partial \left(\lVert w \rVert_s^2\right)$.
When $s=n$, we have $\omega_{\mathbb C^n,n} = \omega_{\mathbb C^n}$.
In what follows we write $\mathbb C^n_s=\mathbb C^n$ to indicate that we have equipped $\mathbb C^n$ with the indefinite K\"ahler form $\omega_{\mathbb C^n,s}$.
Here, $(\mathbb C^n_s,\omega_{\mathbb C^n,s})$ is of constant holomorphic sectional curvature $0$. We also write $\omega_{\mathbb C^n_s}=\omega_{\mathbb C^n,s}$ for convenience.

Define an open subset of $\mathbb P^n$ by
\[ \mathbb P^n_s:=\left\{ [z_0,\ldots,z_n]\in \mathbb P^n : \sum_{j=0}^s |z_j|^2 - \sum_{j=s+1}^n |z_j|^2 > 0 \right\} \]
and $\omega_{\mathbb P^n_s}$ is defined so that
$\omega_{\mathbb P^n_s}:=\sqrt{-1} \partial \overline\partial \log \left(1+\lVert w \rVert_s^2\right)$
on $U_0 \cap \mathbb P^n_s$, where $U_0=\{[z_0,\ldots,z_n]\in \mathbb P^n : z_0\neq 0\}$, $w_j:=z_j/z_0$ for $j=1,\ldots,n$, and $w=(w_1,\ldots,w_n)$.
When $s=n$, $\mathbb P^n_n = \mathbb P^n$.
Here, $(\mathbb P^n_s,\omega_{\mathbb P^n_s})$ is of constant holomorphic sectional curvature $2$.

Define an open subset of $\mathbb C^n$ by
\[ \mathbb B^n_s:=\left\{ (w_1,\ldots,w_n)\in \mathbb C^n: \lVert w \rVert_s^2 < 1 \right\} \]
and $\omega_{\mathbb B^n_s}:=-\sqrt{-1} \partial \overline\partial \log \left(1-\lVert w \rVert_s^2\right)$.
When $s=n$, we have $\mathbb B^n_n = \mathbb B^n$.
Here, $(\mathbb B^n_s,\omega_{\mathbb B^n_s})$ is of constant holomorphic sectional curvature $-2$.

In 1982, M. Barros and A. Romero \cite[Theorem 3.4, p.\,57]{BR82} proved that a simply connected indefinite complex space form of complex dimension $n$ is holomorphically isometric to $\mathbb C^n_s$, $\mathbb B^n_s$ or $\mathbb P^n_s$ for some $s$, $0\le s\le n$.
We say that two simply connected indefinite complex space forms $M$ and $N$ are of the same type if and only if $M=M(c)$ and $N=N(c')$ such that either $c,c'\neq 0$ are of the same sign, or $c=c'=0$; otherwise, we say that $M$ and $N$ are of different types.

For $1\le j,k\le n$, we define
\[ \delta_{jk;s} := \begin{cases}
\delta_{jk} & \text{if } 1\le j,k\le s, \\
-\delta_{jk} & \text{if } s+1\le j,k\le n.
\end{cases} 
\quad\text{and}\quad
a_{jk;s} := \begin{cases}
1 & \text{if } 1\le j,k\le s \text{ or } s+1\le j,k\le n,\\
-1 & \text{otherwise}.
\end{cases} \]
By direct computations we have
$\omega_{\mathbb C^n_s}
=\sqrt{-1}\left(
\sum_{j=1}^s dw_j\wedge d\overline{w_j} - \sum_{j=s+1}^n dw_j\wedge d\overline{w_j}\right)$,
\[ \omega_{\mathbb B^n_s}={\sqrt{-1}\over (1-\lVert w\rVert_s^2)^2}
\sum_{j,k=1}^n \left( (1-\lVert w\rVert_s^2) \delta_{jk;s}
 + a_{jk;s} \overline{w_j}w_k\right) dw_j\wedge d\overline{w_k} \]
and
\[ \omega_{\mathbb P^n_s}|_{U_0\cap \mathbb P^n_s}={\sqrt{-1}\over (1+\lVert w\rVert_s^2)^2}
\sum_{j,k=1}^n \left( (1+\lVert w\rVert_s^2) \delta_{jk;s}
 - a_{jk;s} \overline{w_j}w_k\right) dw_j\wedge d\overline{w_k} \]
For $M=\mathbb C^n_s$, $\mathbb B^n_s$ or $\mathbb P^n_s$, we also write $\omega_M= \sqrt{-1}\sum_{j,k=1}^n (g_M)_{j\overline k} dw_j \wedge d\overline{w_k}$ in terms of the standard (local) holomorphic coordinates $(w_1,\ldots,w_n)$. The reader may refer to \cite{Um88, CDY17} for detailed description of indefinite complex space forms.

By the similar proof, we have the following generalization of Theorem \ref{thm:M_to_Cn_pp form} to indefinite complex space forms by using the Umehara algebra. However, we do not know whether the general case is true or not.

\begin{pro}\label{thm:M_to_Cn_t_pp form}
Let $M$ be either $\mathbb B^m_s$ or $\mathbb P^m_s$ equipped with the indefinite K\"ahler form $\omega_M$ as above, where $0\le s \le m$.
We fix some point $x_0\in M$.
Let $t$ be an integer such that $0\le t\le n$.
Then, there does not exist a germ of holomorphic map $F:(M;x_0) \to (\mathbb C^n_t;F(x_0))$ such that
\begin{equation}\label{Eq:M_to_C_t_pp}
F^*\omega_{\mathbb C^n_t}^p = \lambda \omega_{M}^p
\end{equation}
where $\lambda\neq 0$ is a real constant and $p$ is an integer such that $1\le p\le \min\{m,n\}$.
\end{pro}

\begin{Problem}\label{Problem:Ind_CSF}
Let $M$ and $N$ be indefinite complex space forms of finite complex dimensions $m$ and $n$ respectively, i.e., $(M,\omega_M)$ and $(N,\omega_N)$ are connected complete indefinite K\"ahler manifolds of constant holomorphic sectional curvatures.
If $M$ and $N$ are of different types, then there does not exist a local holomorphic map $F:(M;x_0) \to (N;F(x_0))$ such that $F^*\omega_N^p = \lambda \omega_M^p$ holds for some real constant $\lambda\neq 0$ and some integer $p$, $1\le p\le \min\{m,n\}$.
\end{Problem}

\subsection{Generalization of relatives for K\"ahler manifolds}\label{Sec:Relat}

In 2010, A. J. Di Scala and A. Loi \cite[p.\,496]{DL10} introduced the concept of relatives for K\"ahler manifolds. 
For $(p,p)$-forms, we give the following definition that generalizes the concept of relatives for K\"ahler manifolds.

\begin{defn}\label{defn:relatives pp type}
Let $(N_j,\omega_{N_j})$ be a K\"ahler manifold of complex dimension $n_j\ge 1$ for $j=1,2$.
Let $m$ and $p$ be integers such that $1\le p \le \min\{n_1,n_2,m\}$.
We say that $(N_1,\omega_{N_1})$ and $(N_2,\omega_{N_2})$ are $m$-relatives of type $(p,p)$ if and only if there exist a K\"ahler manifold $(M,\omega_M)$ of complex dimension $m \ge 1$, and germs of non-constant holomorphic maps $F:(M;x_0)\to (N_1;F(x_0))$ and $G:(M;x_0)\to (N_2;G(x_0))$ such that
\begin{equation}\label{Eq:m_relatives_pp}
F^*\omega_{N_1}^p = \lambda_1 \omega_M^p
\quad\text{and}\quad
G^*\omega_{N_2}^p = \lambda_2 \omega_M^p
\end{equation}
hold for some non-zero real constants $\lambda_1,\lambda_2$.
We say that $(N_1,\omega_{N_1})$ and $(N_2,\omega_{N_2})$ are general $m$-relatives of type $(p,p)$ if and only if there exist a complex manifold $M$ of complex dimension $m \ge 1$, and germs of non-constant holomorphic maps $F:(M;x_0)\to (N_1;F(x_0))$ and $G:(M;x_0)\to (N_2;G(x_0))$ such that
\begin{equation}\label{Eq:general_m_relatives_pp}
F^*\omega_{N_1}^p = \lambda G^*\omega_{N_2}^p 
\end{equation}
holds for some real constant $\lambda\neq 0$.
\end{defn}

\begin{rem}\text{}
\begin{enumerate}
\item It is obvious that if the K\"ahler manifolds $(N_1,\omega_{N_1})$ and $(N_2,\omega_{N_2})$ are $m$-relatives of type $(p,p)$, then they are general $m$-relatives of type $(p,p)$.
\item If the K\"ahler manifolds $(N_1,\omega_{N_1})$ and $(N_2,\omega_{N_2})$ are $m$-relatives with a common complex $m$-dimensional submanifold $M$ in the sense of Di Scala-Loi {\rm(}cf.\,\cite[p.\,496]{DL10}{\rm)}, then they are $m$-relatives of type $(1,1)$, and hence they are $m$-relatives of type $(p,p)$ for $1\le p\le \min\{n_1,n_2,m\}$.
\item On the other hand, if the K\"ahler manifolds $(N_1,\omega_{N_1})$ and $(N_2,\omega_{N_2})$ are $m$-relatives of type $(p,p)$ for some integer $p$ with $1 \leq p < m$, then by Proposition \ref{thm:pp_form_to_Holo_iso} they are $m$-relatives of type $(1,1)$.
\end{enumerate}
\end{rem}

By the same idea of the proof of Theorem \ref{thm:General_M_N_curva assump}, we have the following result.

\begin{pro}\label{thm:no relatives for diff signs of curva}
Let $(N_j,\omega_{N_j})$ be a K\"ahler manifold of complex dimension $n_j\ge 1$ for $j=1,2$, such that $(N_1,\omega_{N_1})$ is of negative holomorphic bisectional curvature, $(N_2,\omega_{N_2})$ is of semipositive holomorphic bisectional curvature, and $n_1 \ge n_2$.
Fix a complex manifold $M$ of complex dimension $m \ge n_2$.
Then, there does not exist germs of non-constant holomorphic maps $F:(M;x_0)\to (N_1;F(x_0))$ and $G:(M;x_0)\to (N_2;G(x_0))$ such that
\begin{equation}\label{Eq:General_no relatives for diff signs of curva}
F^*\omega_{N_1}^{n_2} = \lambda G^*\omega_{N_2}^{n_2}
\end{equation}
holds for some real constant $\lambda>0$.
In other words, $(N_1,\omega_{N_1})$ and $(N_2,\omega_{N_2})$ are not general $m$-relatives of type $(n_2,n_2)$ for any integer $m\ge n_2$.
\end{pro}

\begin{proof}
Assume the contrary that there exist such germs of non-constant holomorphic maps $F$ and $G$.
Write $U$ for a connected open neighborhood of $x_0$ in $M$ that is contained in the intersection of the domains of $F$ and $G$. 
We define the unit sphere bundles $S_1$ and $S_2$ in $\wedge^{n_2} T_{N_1}$ and $\wedge^{n_2} T_{N_2}$ with defining functions $\rho_1$ and $\rho_2$ respectively given by $\rho_j(z_j,\xi_j):=\lambda_j\omega_{N_j}^{n_2}(z_j)(\xi_j,\overline{\xi_j}) - 1$ for $j=1,2$, where $\lambda_1:=1$ and $\lambda_2:=\lambda$, i.e., $S_j:=\{(z_j,\xi_j)\in \wedge^{n_2} T_{N_j}:\rho_j(z_j,\xi_j)=0\}$ for $j=1,2$.
Moreover, $F$ and $G$ induce the holomorphic maps $f:=(F,dF):\wedge^{n_2} T_U \to \wedge^{n_2} T_{N_1}$ and $g:=(G,dG):\wedge^{n_2} T_U \to \wedge^{n_2} T_{N_2}$.
By (\ref{Eq:General_no relatives for diff signs of curva}) and the definitions of $\rho_1$ and $\rho_2$, we have $\rho_1\circ f = \rho_2 \circ g=:\rho$, which defines a real hypersurface $S$ in $\wedge^{n_2} T_U$.
By applying $\partial\overline\partial$ to both sides of $\rho_1\circ f = \rho_2 \circ g$, for $(w,v)\in S\subset \wedge^{n_2} T_U$ we have
\begin{equation}\label{Eq:same Pull back Levi form_1}
(\partial\overline\partial \rho_1)(f(w,v))(f_*(\eta),\overline{f_*(\eta)}) = (\partial\overline\partial \rho_2)(g(w,v))(g_*(\eta),\overline{g_*(\eta)})
\end{equation}
for all $\eta\in T^{1,0}_{(w,v)}(S)$.
By similar computations in Section \ref{Sec:USB_Pm} and in Yuan \cite{Yu17}, we have the following.
\begin{enumerate}
\item[(a)] $S_1$ is strictly pseudoconvex.
\item[(b)] The holomorphic tangent space $T^{1,0}_{g(w,v)}(S_2)$ consists of vectors of the form $\eta'=(\eta'_1,{\bf 0})$ with $\eta'_1\in T_{G(w)}(N_2)$, and $(\partial\overline\partial \rho_2)(g(w,v))(\eta',\overline{\eta'})\le 0$.
\end{enumerate}
Since $F$ is non-constant, we may choose $w\in U$ and $\eta=(\eta_1,{\bf 0})\in T^{1,0}_{(w,v)}(S)$ with $\eta_1\in T_w(U)$ such that $f_*(\eta)\neq {\bf 0}$.
By (a) we have $(\partial\overline\partial \rho_1)(f(w,v))(f_*(\eta),\overline{f_*(\eta)})>0$.
On the other hand, by (b) we have 
$(\partial\overline\partial \rho_2)(g(w,v))(g_*(\eta),\overline{g_*(\eta)})\le 0$.
This contradicts with (\ref{Eq:same Pull back Levi form_1}). The result follows.
\end{proof}

Applying the previous proposition to complex space forms, we have the following corollary. We omit the proof as the idea of the proof is similar. However, the general situation is still open.

\begin{cor}\label{cor:no relatives for different types}
Let $(N_j,\omega_{N_j})$, $j=1,2$, be either $\mathbb C^n$, $\mathbb P^n$ or $\mathbb B^k$ for some integer $k\ge n$ such that $N_1$ and $N_2$ are of different types.
Then, $(N_1,\omega_{N_1})$ and $(N_2,\omega_{N_2})$ are not general $m$-relatives of type $(n,n)$ for any integer $m\ge n$.
\end{cor}

The following problem can be viewed as a generalization of Problem \ref{Problem:Ind_CSF}.

\begin{Problem}\label{Problem:CSF_diff_type_not relati pp}
If $(N_1,\omega_{N_1})$ and $(N_2,\omega_{N_2})$ are {\rm(}possibly indefinite{\rm)} complex space forms of different types and of finite complex dimensions $n_1$ and $n_2$ respectively, then are $(N_1,\omega_{N_1})$ and $(N_2,\omega_{N_2})$ not {\rm(}general{\rm)} $m$-relatives of type $(p,p)$ for any integers $m$ and $p$ with $1\le p \le \min\{n_1,n_2,m\}$?
\end{Problem}

\subsection{Local holomorphic maps between complex space forms of the same type preserving $(p,p)$-forms}
By Proposition \ref{thm:pp_form_to_Holo_iso_0}, we only need to consider the local holomorphic maps from $M^m$ to $N^n$ preserving $(m,m)$-forms.
We first observe that in general there is no global extension and rigidity theorem for local holomorphic maps from $\mathbb C^m$ into $\mathbb C^n$ preserving $(m,m)$-forms as indicated in the following example.

\begin{eg}\label{Remark:CE_Cm_to_Cn}
Define $F(z_1,\ldots,z_m):= \left(z_1+{1\over 1-z_m},z_2,\ldots,z_m,0,\ldots,0\right)$.
Then, we have $F^*\omega_{\mathbb C^n}^m=\omega_{\mathbb C^m}^m$ but $F^*\omega_{\mathbb C^n} \neq r \omega_{\mathbb C^m}$ for any real constant $r>0$.
\end{eg}

On the other hand, it was proved by the second author that any local holomorphic map from $\mathbb{B}^m$ to $\mathbb{B}^n$ preserving $(p, p)$-forms with $1\leq p \leq \min\{n, m\}$ is totally geodesic \cite{Yu17}. It remains to consider the case where $M=\mathbb P^m$ and $N=\mathbb P^n$. 
Note that the Veronese embedding from $\mathbb P^n$ to $\mathbb P^{2n+1}$ is a non-totally geodesic holomorphic isometric embedding. We now only wonder whether the local holomorphic map from $\mathbb{P}^m$ to $\mathbb{P}^n$ preserving $(m, m)$-forms is a holomorphic isometric embedding.
We obtain the following partial result for the global extension and the metric rigidity problem.

\begin{pro}\label{Pro:Partial ext and rigidity_P_to_P}
Let $F:(\mathbb P^m;x_0)\to (\mathbb P^n;F(x_0))$ be a germ of holomorphic map such that $F^*\omega_{\mathbb P^n}^m = \lambda \omega_{\mathbb P^m}^m$ for some real constant $\lambda>0$.
Then, we have $\lambda \ge 1$.
If we assume in addition that $n=m$ or $\lambda=1$, then $F$ extends to a holomorphic isometry from $({\mathbb P^m},\omega_{\mathbb P^m})$ to $({\mathbb P^n},\omega_{\mathbb P^n})$.
\end{pro}
\begin{proof}
Note that $(\mathbb P^n, \omega_{\mathbb P^n})$ is of constant holomorphic sectional curvature $2$ and ${\rm Ric}(\mathbb P^m,\omega_{\mathbb P^m})=(m+1) \omega_{\mathbb P^m}$.
Let $U$ be an open neighborhood of $x_0$ where $F^*\omega_{\mathbb P^n}^m = \lambda \omega_{\mathbb P^m}^m$ holds. 
Then, we have
\[ F^*{\rm Ric}(F(U),\omega_{\mathbb P^n}|_{F(U)})
={\rm Ric}(U,F^*\omega_{\mathbb P^n})
={\rm Ric}(U,\omega_{\mathbb P^m}|_U)
= (m+1) \omega_{\mathbb P^m}|_U. \]
By the Gauss equation for the complex submanifold $(F(U),\omega_{\mathbb P^n}|_{F(U)}) \subset ({\mathbb P^n},\omega_{\mathbb P^n})$, we have
\[ {\rm Ric}(F(U),\omega_{\mathbb P^n}|_{F(U)}) \le (m+1) \omega_{\mathbb P^n}|_{F(U)} \]
(cf. \cite[p.\,177]{KN69}).
Therefore, we have
\begin{equation}\label{Eq:Pullback_Kahler_form}
(m+1) \omega_{\mathbb P^m}|_U \le (m+1) F^*\omega_{\mathbb P^n},
\end{equation}
so that $\omega_{\mathbb P^m}^m|_U \le  F^*\omega_{\mathbb P^n}^m = \lambda \omega_{\mathbb P^m}^m|_U$ and thus $\lambda \ge 1$.
If $\lambda=1$, then by $F^*\omega_{\mathbb P^n}-\omega_{\mathbb P^m}|_U \ge 0$ and $F^*\omega_{\mathbb P^n}^m - \omega_{\mathbb P^m}^m|_U = 0$, we must have $F^*\omega_{\mathbb P^n}-\omega_{\mathbb P^m}|_U=0$ so that $F$ is a local holomorphic isometry from $(\mathbb P^m,\omega_{\mathbb P^m})$ into $(\mathbb P^n,\omega_{\mathbb P^n})$.
It follows that $F$ extends to a totally geodesic holomorphic isometry from $(\mathbb P^m,\omega_{\mathbb P^m})$ into $(\mathbb P^n,\omega_{\mathbb P^n})$ by Calabi \cite{Ca53} and the Gauss equation.

Suppose $m=n$. Note that $F^*\omega_{\mathbb P^n}^n = \lambda \omega_{\mathbb P^n}^n$ is equivalent to
\[ \det\left((g_{\mathbb P^n})_{i\overline j}(F(w))\right)_{1\le i,j\le n} |\det JF(w)|^2 = \det\left((g_{\mathbb P^n})_{i\overline j}(F(w))\right)_{1\le i,j\le n} \]
in terms of the local holomorphic coordinates $w=(w_1,\ldots,w_n)$ around $x_0$ on $\mathbb P^n$.
Since ${\rm Ric}({\mathbb P^k},\omega_{\mathbb P^k})$ $=(k+1)\omega_{\mathbb P^k}$ for all integer $k\ge 1$, we apply $\partial\overline\partial\log$ on both sides of the above identity and obtain
\[ (n+1)F^*\omega_{\mathbb P^n}
=F^*{\rm Ric}({\mathbb P^n},\omega_{\mathbb P^n})
= {\rm Ric}(U,\omega_{\mathbb P^n}|_U)
= (n+1) \omega_{\mathbb P^n}|_U, \]
i.e., $F^*\omega_{\mathbb P^n}=\omega_{\mathbb P^n}$ on $U$, and $F$ is a local holomorphic isometry from $({\mathbb P^n},\omega_{\mathbb P^n})$ to $({\mathbb P^n},\omega_{\mathbb P^n})$.
Hence, by Calabi \cite{Ca53} $F$ extends to a (global) holomorphic isometry from $({\mathbb P^n},\omega_{\mathbb P^n})$ to $({\mathbb P^n},\omega_{\mathbb P^n})$, and thus $\lambda=1$.
\end{proof}

We observe the following metric rigidity theorem for global holomorphic maps from a compact K\"ahler manifold $(M,\omega_M)$ with the second Betti number $b_2(M)=1$ into a K\"ahler manifold $(N,\omega_N)$ of complex dimension $n$ preserving $(m,m)$-forms, where $m:=\dim_{\mathbb C}(M)\ge 1$.

\begin{pro}\label{pro:Global holo map_M_CptKah_Picard_number 1_to_N_mm forms}
Let $(M,\omega_M)$ be a compact K\"ahler manifold with the second Betti number $b_2(M)=1$ and $m:=\dim_{\mathbb C}(M)\ge 1$, and let $(N,\omega_N)$ be a K\"ahler manifold of complex dimension $n$.
If $F:M\to N$ is a global holomorphic map such that $F^*\omega_{N}^m = \lambda \omega_{M}^m$ for some real constant $\lambda>0$, then $F:(M,\lambda^{1\over m}\omega_{M})\to (N,\omega_{N})$ is a holomorphic isometry.
\end{pro}
\begin{proof}
By $F^*\omega_{N}^m = \lambda \omega_{M}^m$, $F:M\to N$ is of rank $m$, i.e., $F$ is an immersion, and $(F^*\omega_{N})^m = (\lambda^{1\over m}\omega_{M})^m$ so that ${\rm Ric}(M,F^*\omega_{N}) = {\rm Ric}(M,\lambda^{1\over m}\omega_{M})$.
In addition, we have $F^*\omega_{N}\ge 0$.
Since $M$ has the second Betti number $b_2(M)=1$, we have $H^2(M,\mathbb R)\cong \mathbb R$, and thus we have the same cohomology class $[F^*\omega_{N}]=[\lambda^{1\over m}\omega_{M}]$ in $H^2(M,\mathbb R)$.
By the $\partial\overline\partial$-lemma we have
\[ F^*\omega_{N} = \lambda^{1\over m}\omega_{M} + \sqrt{-1}\partial\overline\partial \varphi \]
for some real smooth function $\varphi$ on $M$.
Writing $\omega:=F^*\omega_{N}$, we have
\[ 0 =\omega^m - (\lambda^{1\over m}\omega_{M})^m
= \sqrt{-1}\partial\overline\partial \varphi \wedge \left(\sum_{j=0}^{m-1} \omega^j \wedge (\lambda^{1\over m}\omega_{M})^{m-1-j}\right). \]
Denoting by $\Xi:=\sum_{j=0}^{m-1} \omega^j \wedge (\lambda^{1\over m}\omega_{M})^{m-1-j}$, we have $\Xi \ge \omega^{m-1}$ and $d\Xi=0$.
Therefore, we have
\[ d( \sqrt{-1}\varphi \cdot \overline\partial \varphi \wedge \Xi)
= \sqrt{-1} \partial \varphi \wedge \overline\partial \varphi \wedge \Xi
 + \varphi \cdot \sqrt{-1}\partial\overline\partial \varphi \wedge \Xi \]
and thus by Stokes' Theorem,
\[ \begin{split}
0 =& - \int_{M} \varphi \left(\omega^m - (\lambda^{1\over m}\omega_{M})^m\right)
= -\int_{M} \varphi \cdot \sqrt{-1}\partial\overline\partial \varphi \wedge \Xi
= \int_{M} \sqrt{-1}\partial \varphi \wedge \overline\partial \varphi \wedge \Xi\\
\ge & \int_{M} \sqrt{-1}\partial \varphi \wedge \overline\partial \varphi \wedge \omega^{m-1} 
= \int_{M} {1\over m}|\partial \varphi|_{\omega}^2 \cdot \omega^m \ge 0. 
\end{split}\]
Therefore, $\varphi$ must be a constant function and thus $F^*\omega_{N} = \lambda^{1\over m}\omega_{M}$, i.e., $F:(M,\lambda^{1\over m}\omega_{M})\to (N,\omega_{N})$ is a holomorphic isometry, as desired.
\end{proof}
\begin{rem}
Note that Proposition \ref{pro:Global holo map_M_CptKah_Picard_number 1_to_N_mm forms} also applies to all Fano manifolds $M$ of Picard number $1$ because $H^2(M,\mathbb Z)\cong {\rm Pic}(M)\cong \mathbb Z$.
Here, a Fano manifold $X$ is a projective manifold such that the anticanonical bundle $-K_X$ is ample.
For example, any irreducible Hermitian symmetric space of compact type is a Fano manifold of Picard number $1$.
As a consequence, if $F:\mathbb P^m \to \mathbb P^n$ is a {\rm(}global{\rm)} holomorphic map such that $F^*\omega_{\mathbb P^n}^m = \lambda \omega_{\mathbb P^m}^m$ for some real constant $\lambda>0$, and $n, m\ge 1$, then $F:(\mathbb P^m,\lambda^{1\over m}\omega_{\mathbb P^m})\to (\mathbb P^n,\omega_{\mathbb P^n})$ is a holomorphic isometry.
\end{rem}

\end{document}